\documentclass[a4paper,11pt,reqno]{amsart}
\usepackage{amsmath,amsfonts,amssymb,amsthm,enumerate,multicol}
\usepackage{tikz,graphicx}
\usepackage{subfigure,hyperref}
\usepackage{caption,enumitem,wasysym}
\usepackage{cleveref}
\usepackage{longtable}
\usepackage{epstopdf}
\usepackage{float,wrapfig}
\usepackage[labelsep=space]{caption}
\usepackage[font=footnotesize]{caption}
\usepackage{booktabs} 
\usepackage{array}    
\usepackage{multirow} 

\usepackage{xcolor}
\newtheorem{lemma}{Lemma}
\newtheorem{theorem}{Theorem}[section]

\newtheorem{example}{Example}[section]
\newtheorem{corollary}{Corollary}[section]
\newtheorem{remark}{Remark}[section]

\allowdisplaybreaks
\DeclareMathOperator{\real}{Re}

\allowdisplaybreaks

\begin{document}

\title{On Lemniscate starlikeness of analytic functions and its application to special functions}

\author{Saiful R. Mondal, Ahmad K. Al Abdulaali}
\address{Department of Mathematics and Statistics, Collage of Science, King Faisal University, Al-Hasa 31982, Hofuf, Saudi Arabia}



\begin{abstract}
This paper investigates the lemniscate starlikeness of analytic functions by deriving specific conditions on their power series coefficients. The study utilizes the Cauchy product of power series along with key inequalities involving the Pochhammer symbol and the Gamma function. The derived results are further applied to a number of special functions, providing parameter restrictions under which these functions become lemniscate starlike. The findings extend and refine several earlier results in this domain.  
\end{abstract}


\maketitle


\section{Introduction}\label{intro}
We consider the analytic and univalent functions $f$ defined in the unit disk $\mathcal{U} = \{ z \in \mathbb{C} : |z| < 1 \}$ 
and normalized by $f(0) = 0, \quad f'(0) = 1$. We denote such classes by $\mathcal{S}$. For a function $f \in \mathcal{S}$ have the power series form 
\begin{align}\label{eqn:power-series}
    f(z)= z + \sum_{k=2}^{\infty} a_k z^k.
\end{align}

Given $0 \leq \alpha < 1$, $\mathcal{S}^*(\alpha)$, the class consisting of functions $f \in \mathcal{S} $ with analytic charaterization 
\[
\real \left\{ \frac{z f'(z)}{f(z)} \right\} > \alpha \quad (z \in \mathcal{U}),
\]
is known as starlike of order $\alpha$. Similarly, $f$ is said to be in the class of functions \emph{convex of order $\alpha$}, denoted by $K(\alpha)$, if
\[
\real \left\{ 1 + \frac{z f''(z)}{f'(z)} \right\} > \alpha \quad (z \in \mathcal{U}).
\]
 For two analytic functions \(f\) and \(g\) defined on \(\mathbb{D}\), the function \(f\) is said to be \textit{subordinate} to \(g\), denoted by \(f \prec g\), if there exists an analytic function \(w: \mathbb{D} \to \mathbb{D}\) satisfying \(w(0) = 0\) and \(f(z) = g(w(z))\). Moreover, if the function \(g\) is univalent, then \(f \prec g\) if and only if \(f(0) = g(0)\) and \(f(\mathbb{D}) \subset g(\mathbb{D})\).
	
	Consider the function \(\phi_l(z) := \sqrt{1+z}\), which maps \(\mathbb{D}\) univalently onto the region illustrated in Figure \ref{fig:lem-disc}. This region is commonly referred to as the \textit{lemniscate domain}; see \cite{sokol1996radius,sokol2018further}. 

\begin{figure}
    \centering
    \includegraphics[width=0.5\linewidth]{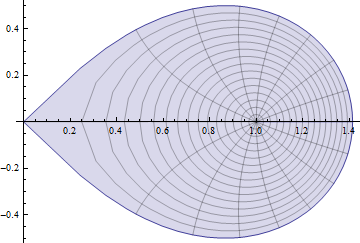}
    \caption{Image of $\phi_l(\mathbb{D})$. }
    \label{fig:lem-disc}
\end{figure}


A function $f \in \mathcal{S}$ is lemniscate starlike if  $\frac{z f'(z)}{f(z)} \prec \sqrt{1+  z}$. This is equivalent to say 
\begin{align}
  \left| \frac{(z f'(z))^2}{(f(z))^2} -1 \right|< 1. 
\end{align}
The lemniscate starlike can be further refined by considering $\frac{z f'(z)}{f(z)} \prec \sqrt{1+ c z}$, $0< c \le 1$. We denote this class as $\mathcal{S}^{\ast}(q_c)$. The class  $\mathcal{S}^{\ast}(q_c)$ is the main attraction of our study. This class was introduced and studied in \cite{sokol2018further}.

The \emph{cross-product} (Cauchy product) of the power series 
\[
g_1(z) = \sum_{n=0}^{\infty} a_n z^n 
\quad \text{and} \quad 
g_2(z) = \sum_{n=0}^{\infty} b_n z^n,
\]
which serves as the main analytical tool in this study, is defined by
\[
g_3(z) = \sum_{n=0}^{\infty} c_n z^n, 
\quad \text{where} \quad 
c_n = \sum_{j=0}^{n} a_j b_{\,n-j}.
\]

In this study, we derive the conditions on the coefficient \( a_k \) of the power series \eqref{eqn:power-series} for a function \( f \in \mathcal{S} \) such that \( f \in \mathcal{S}^{\ast}(q_c) \), where \( c \in (0,1] \). The main result in this regard is presented in \Cref{sec:Main-result}. This result is further applied to several special functions, and the corresponding outcomes are discussed in various subsections of \Cref{sec:Appl-special}. Specifically, \Cref{sec:Appl-exponential} provides examples involving exponential functions, while \Cref{sec:Appl-conf-hyp} addresses the case of confluent hypergeometric functions. The lemniscate starlikeness of the Bessel and Struve functions is analyzed in \Cref{sec:Appl-Bessel} and \Cref{sec:Appl-Struve}, respectively. The result related to Bessel functions refines and extends the findings reported in \cite{Zayed-Bulboacua}. The product of Bessel functions is examined in \Cref{sec:Appl-CP-Bessel}, and an illustrative example involving the error function is discussed in \Cref{sec:Appl-error}. For clarity and readability, a brief literature review of each special function under consideration is included in its respective subsection.

Moreover, several fundamental results involving Pochhammer symbols and Gamma functions, which play an essential role in the development of the main theorem and its applications, are provided in \Cref{sec:fundamental}. Finally, the study concludes by listing several additional special functions that can also be investigated using the established main result. Owing to the similarity of the arguments, detailed computations for these cases are omitted.

\section{Fundamental results}\label{sec:fundamental}
We begin this section by proving certain fundamental results necessary for the framework of our study. Furthermore, we state auxiliary results from the literature that will facilitate the application of the main theorems obtained herein.

\begin{lemma}\label{lemma-basic1}
    For $ n \in N$, 
    \begin{align}\label{eqn:lemma-basic1}
    \frac{n^2}{4}+n\le \left(\frac{21}{4}\right)^{\frac{n}{3}}.
    \end{align}
    The equality holds at $n=3$.
\end{lemma}
\begin{proof}
First, we test the cases numerically for $n=1, 2,$ and $n=3$.
\begin{enumerate}
\item $n=1$: $\frac{n^2}{4}+n= 1.25$ , $\left(\frac{21}{4}\right)^{\frac{n}{3}}= 1.73801$
\item  $n=2$ : $\frac{n^2}{4}+n= 3$ , $\left(\frac{21}{4}\right)^{\frac{n}{3}}= 3.02069$
\item  $n=3$ : $\frac{n^2}{4}+n= 5.25$ , $\left(\frac{21}{4}\right)^{\frac{n}{3}}= 5.25$. 
\end{enumerate}
Clearly, the inequality \eqref{eqn:lemma-basic1} holds for $n=1, 2$,  and both sides are equal for $n=3$. 

Suppose that the inequality \eqref{eqn:lemma-basic1} is true for $m=n \ge 3$, that is, 
\begin{align*}
    \frac{m^2}{4}+m\le \left(\frac{21}{4}\right)^{\frac{m}{3}}.
\end{align*}
Now,
\begin{align*}
    \frac{(m+1)^2}{4}+(m+1)&=\frac{m^2}{4}+ \frac{3 m}{2} +\frac{5}{4}<\left(\frac{21}{4}\right)^{\frac{m}{3}}+\frac{ m}{2}+\frac{5}{4}<\left(\frac{21}{4}\right)^{\frac{m+1}{3}}.
\end{align*}
To prove the last part of the inequality, for real $x \geq 3$, define the function
\[
H(x) = \left(\frac{21}{4}\right)^{\tfrac{x+1}{3}} - \left(\frac{21}{4}\right)^{\tfrac{x}{3}} - \frac{x}{2} - \frac{5}{4}.
\]
Clearly, the inequality $H(x) >0$ is equivalent to showing $H(m) > 0$ for all integers $m \geq 3$.

\medskip
The derivatives of $H(x)$

\[
H'(x) = \frac{\ln(21/4)}{3}\,\left(\frac{21}{4}\right)^{\tfrac{x}{3}}
\left(\left(\frac{21}{4}\right)^{1/3}-1\right) - \frac{1}{2}.
\]

\medskip

Let $C = \left(\tfrac{21}{4}\right)^{1/3} - 1 \approx 0.73$ and 
$B = \tfrac{\ln(21/4)}{3} \approx 0.55$. Then
\[
H'(x) = BC \left(\frac{21}{4}\right)^{x/3} - \frac{1}{2}.
\]
Since $\tfrac{21}{4} > 1$, the term $\left(\tfrac{21}{4}\right)^{x/3}$ grows exponentially with $x$. 
Thus for $x \geq 3$, $H'(x)$ is strictly positive. Therefore $H(x)$ is strictly increasing on $[3,\infty)$.

\medskip

Numerically,
\[
H(3) = \left(\frac{21}{4}\right)^{4/3} - \left(\frac{21}{4}\right) - \frac{11}{4} \approx 1.08 > 0.
\]

\medskip

Since $H(3) > 0$ and $H(x)$ is increasing for $x \geq 3$, it follows that
\[
H(m) > 0 \quad \text{for all integers } m \geq 3.
\]
This proves the claim. 
\end{proof}

\begin{lemma}\label{lemma-basic2}
    For $k \in [1,n]$ for a fixed $2 \le n \in N$, 
    \begin{align}
     n(k+1) -k^2 \le \left(\frac{21}{4}\right)^{\frac{n}{3}}.
    \end{align}
\end{lemma}
\begin{proof}
    First, denote the function
    \begin{align}
        \Psi_n(k):= n(k+1) - k^2.
    \end{align}
    It is straightforward to verify that 
\[
\Psi_n(k) = k(n-k) + n > 0, \quad \forall\, n \in \mathbb{N}, \; k \in [1,n].
\]
Differentiating with respect to $k$ yields
\[
\Psi_n'(k) = n - 2k,
\]
which vanishes at $k = \tfrac{n}{2}$. Furthermore,
\[
\Psi_n''(k) = -2 < 0,
\]
demonstrating that $\Psi$ is strictly concave and hence attains its maximum at $k = \tfrac{n}{2}$. Consequently, the maximal value of $\Psi$ is given by
\begin{equation}
    \Psi_n\!\left(\tfrac{n}{2}\right) = \frac{n^{2}}{4} + n. 
\end{equation}
From \Cref{lemma-basic1}, it follows that 
\[\Psi_n\!(k) = \frac{n^{2}}{4} + n \leq \left(\frac{21}{4}\right)^{\frac{n}{3}}. \]
This completes the proof. \end{proof}

We need the following basic information about the generalized hypergeometric function: The {generalized hypergeometric function} is denoted by
\[
{}_pF_q(a_1, \ldots, a_p; b_1, \ldots, b_q; z)
\]
and is defined by the power series:
\[
{}_pF_q(a_1, \ldots, a_p; b_1, \ldots, b_q; z) = \sum_{n=0}^{\infty} \frac{(a_1)_n \cdots (a_p)_n}{(b_1)_n \cdots (b_q)_n} \cdot \frac{z^n}{n!},
\]
where $(a)_n=a(a+1)\ldots (a+n-1)$ is the well-known Pochammer symbol \cite{Rainville1971}. The convergence of the generalized hypergeometric series depends on the relation between $p$ and $q$:

\begin{itemize}
    \item \textbf{If $p \leq q$:} The series converges for all finite $z \in \mathbb{C}$. Hence, the function is entire.
    
    \item \textbf{If $p = q + 1$:} The series converges for $|z| < 1$. It may also converge at $|z| = 1$ under the condition:
    \[
    \real\left(\sum_{j=1}^q b_j - \sum_{i=1}^p a_i\right) > 0.
    \]
    
    \item \textbf{If $p > q + 1$:} The series diverges for any $z \neq 0$, except in special cases that allow analytic continuation.
\end{itemize}

 Denote $\mathbb{Z}_0^{-}=\{\ldots, -3, -2, -1, 0\}$. The following identity for $b \in \mathbb{R}\setminus \mathbb{Z}_0^{-}$ is useful in the sequel. 
\begin{lemma}\label{lem:identity-1}
For  $b \in \mathbb{R}\setminus \mathbb{Z}_0^{-}$, 
\begin{align}\label{eqn-lem:identity-1}
    \sum_{k=0}^{n} \frac{1}{k!\,(b)_k (n-k)!\,(b)_{n-k}}= \dfrac{2^{2n} \left(b-\tfrac{1}{2}\right)_n}{n! (b)_n (2b-1)_n}.
\end{align}

\end{lemma}
\begin{proof}
    To prove the identity, we need a few basic relations. 
    For $0\le k \le n$,
    \begin{align*}
        &\frac{n!}{(n-k)!}= n(n-1) (n-2)\ldots(n-k+1)= (-1)^k (-n)_k, \\
       & (b)_{n-k}(b+n-k)_k=(b)_n \implies  (b)_{n-k}= \frac{(b)_n}{(b+n-k)_k}.
    \end{align*}
    Thus, we can rewrite the right-hand side of \eqref{eqn-lem:identity-1} as
    \begin{align}
       \sum_{k=0}^{n} \frac{1}{k!\,(b)_k (n-k)!\,(b)_{n-k}}= \frac{1}{n! (b)_n} \sum_{k=0}^{n} \frac{(-1)^k (-n)_k (b+n-k)_k}{(b)_k }. 
    \end{align}
Our next verification is the identity, $(-1)^k  (b+n-k)_k := (1-b-n)_k$ which is trivially true for $k=0, 1$. Now, using mathematical induction along with basic properties of the Pochhammer symbol $(a+1)_m=(a+1) (a)_{m-1}$ and  $(a+m) (a)_m =(a)_{m+1}$, one can easily verify the identity. 

This gives us a closed form of the right-hand side of \eqref{eqn-lem:identity-1} as
    \begin{align}\label{eqn-closed form1}
       \sum_{k=0}^{n} \frac{1}{k!\,(b)_k (n-k)!\,(b)_{n-k}}= \frac{1}{n! (b)_n} \sum_{k=0}^{n} \frac{ (-n)_k (1-b-n)_k}{k! (b)_k }=\frac{{}_2F_{1}(-n, 1-b-n; b; 1)}{n! (b)_n}. 
    \end{align}
The classical Chu--Vandermonde summation formula provides a closed form for a terminating hypergeometric series \cite{Slater1966}:
\[
{}_2F_{1}\!\left(-n,\,\alpha;\,\beta;\,1\right)
=\sum_{k=0}^{n}\frac{(-n)_k\,(\alpha)_k}{(\beta)_k\,k!}
=\frac{(\beta-\alpha)_n}{(\beta)_n}, \quad n \in \mathbb{Z}_{\ge 0}.
\]
Using this closed form \eqref{eqn-closed form1} can be rewrite as 
 \begin{align}\label{eqn-closed form2}
       \sum_{k=0}^{n} \frac{1}{k!\,(b)_k (n-k)!\,(b)_{n-k}}=\frac{(2b-1+n)_n}{n! (b)_n (b)_n}. 
    \end{align}
Next, recall below two identities of the Pochhammer symbol as
\begin{align}\label{eqn-closed form3}
      (a)_{2n}=(a)_n (a+n)_n ; \quad   (a)_{2n}=2^{2n} \left(\frac{a}{2}\right)_n\left(\frac{a+1}{2}\right)_n. 
    \end{align}
Finally, setting $a=2b-1$, we have 
 \begin{align}\label{eqn-closed form2}
       \sum_{k=0}^{n} \frac{1}{k!\,(b)_k (n-k)!\,(b)_{n-k}}=\frac{(2b-1)_{2n} }{n! (2b-1)_{n} (b)_n }=\frac{2^{2n} \left(b-\frac{1}{2}\right)_n}{n! (2b-1)_{n} (b)_n}.
    \end{align}
\end{proof}

\begin{lemma}\label{lem:identity-2}
For  $b \in \mathbb{C}\setminus \mathbb{Z}_0^{-}$ such that $\real(b) \ge -1/2$, then
\begin{align}\label{eqn-lem:identity-2}
    \sum_{k=0}^{n} \left|\frac{1}{k!\,(b)_k}\right| \left|\frac{1}{(n-k)!\,(b)_{n-k}}\right| \le \frac{2^n}{n! |b|^n}.
\end{align}
 \end{lemma}
\begin{proof}
   Suppose $b \in \mathbb{C}\setminus \mathbb{Z}_0^{-}$ with $b= \alpha + i \beta$. It is trivially holds that $|(b)_1|=|b|$. Now, for $k \ge 2$, we have
   \begin{align*}
    |(b)_k|^2= |b|^2|b+1|^2 \ldots |b+k-1|^2
    &= (\alpha^2+\beta^2) ((\alpha+1)^2+\beta^2) \ldots  ((\alpha+k-1)^2+\beta^2) \ge  (\alpha^2+\beta^2)^k =(|b|^k)^2  
   \end{align*} 
   provided $\alpha= \real(b) \ge -1/2$. This together, we have $|(b)_k| \ge |b|^k $ for $k \ge 1$. Similarly, it follows that $|(b)_{n-k}| \ge |b|^{n-k} $ for $k= 0, 1, \ldots, n$. 
   
  Now, from the binomial theorem, we know
   \begin{align*}
       \sum_{k=0}^n \frac{n!}{k! (n-k)!} x^k = (1+x)^n \implies \sum_{k=0}^n \frac{1}{k! (n-k)!}  = \frac{2^n}{n!}. 
   \end{align*}
Finally, we conclude that for $\real(b)>-1/2$
\begin{align*}
    \sum_{k=0}^{n} \left|\frac{1}{k!\,(b)_k}\right| \left|\frac{1}{(n-k)!\,(b)_{n-k}}\right| \le  \sum_{k=0}^{n} \frac{1}{k!\,\left|b\right|^k} \frac{1}{(n-k)!\,\left|b\right|^{n-k}}= \frac{1}{|b|^n} \sum_{k=0}^{n} \frac{1}{k!}\, \frac{1}{(n-k)!} = \frac{2^n}{n! |b|^n}.
\end{align*}
This completes the proof.   
\end{proof}

\section{Lemniscate starlike of analytic functions}\label{sec:Main-result}
Now we state and prove our main theorem.

\begin{theorem}\label{thm:anlytic-LS}
    Let $\{a_n\}_{n\ge 1}$ be real or complex sequence such that 
    \begin{align}\label{eqn:hypo-1-thm1}
a_1=1 \quad \text{and} \quad \sum_{n=1}^{\infty} \sum_{k=0}^n |a_{k+1}|  |a_{n-k+1}|<1.
\end{align}
Suppose that $f(z):= z+ \sum_{n=2}^{\infty} a_n z^n$. Now, for $c \in [0, 1]$ if
\begin{align}\label{eqn:hypo-thm1}
\sum_{n=1}^{\infty} \left( \left(\tfrac{21}{4}\right)^{\tfrac{n}{3}}- c\right) \sum_{k=0}^n  |a_{k+1}|  |a_{n-k+1}| < c, 
\end{align}
then $z f'(z)/ f(z)\prec \sqrt{1+ cz}$. 
\end{theorem}
\begin{proof}
    Now, if $f(z)= z+ \sum_{n=2}^{\infty} a_n z^n$, then 
\begin{align}\label{eqn:dervative-cauchy}\notag
  ( z f'(z) )^2 = \left(z+ \sum_{n=2}^{\infty} n a_n z^n\right)^2 
  &= z^2\left( 1+ \sum_{n=1}^{\infty} (n+1) a_{n+1} z^n\right)^2\\
  &= z^2 \left( 1+ \sum_{n=1}^{\infty} \left(\sum_{k=0}^n (k+1) (n-k+1) a_{k+1}  a_{n-k+1}\right)z^n\right).
\end{align}
Similarly,
\begin{align}\label{eqn:cauchy-product}\notag
  ( f(z) )^2 = \left(z+ \sum_{n=2}^{\infty}  a_n z^n\right)^2 
  &= z^2\left( 1+ \sum_{n=1}^{\infty}  a_{n+1} z^n\right)^2\\
  &= z^2 \left( 1+ \sum_{n=1}^{\infty} \left(\sum_{k=0}^n  a_{k+1}  a_{n-k+1}\right)z^n\right).
\end{align}
From \eqref{eqn:dervative-cauchy} and \eqref{eqn:cauchy-product}, it follwos that 
\begin{align}\label{eqn:dervative-cauchy-2}\notag
  ( z f'(z) )^2- ( f(z) )^2 = z^2  \sum_{n=1}^{\infty} \left(\sum_{k=0}^n \left(n(k+1) -k^2\right) a_{k+1}  a_{n-k+1}\right)z^n.
\end{align}
Thus for $|z|<1$ in together with \Cref{lemma-basic2}, we have
\begin{align*}
   \left|\frac{( z f'(z) )^2- ( f(z) )^2} {z^2}\right| < \sum_{n=1}^{\infty} \sum_{k=0}^n \left|n(k+1) -k^2\right| |a_{k+1}|  |a_{n-k+1}|< \sum_{n=1}^{\infty}\bigg[\left(\frac{21}{4}\right)^{\tfrac{n}{3}} \sum_{k=0}^n  |a_{k+1}|  |a_{n-k+1}|\bigg].
\end{align*}
Simpliarly
\begin{align*}
   \left|\frac{( f(z) )^2} {z^2}\right| >1- \sum_{n=1}^{\infty} \sum_{k=0}^n |a_{k+1}|  |a_{n-k+1}|.
\end{align*}
Finally, we have 
\begin{align*}
   \left| \frac{(z f'(z))^2}{(f(z))^2} -1\right| =  \frac{\left|\tfrac{( z f'(z) )^2- ( f(z) )^2} {z^2}\right|}{\left|\tfrac{( f(z) )^2} {z^2}\right|} &< \dfrac{\sum_{n=1}^{\infty} \sum_{k=0}^n \left|n(k+1) -k^2\right| |a_{k+1}|  |a_{n-k+1}|}{1- \sum_{n=1}^{\infty} \sum_{k=0}^n |a_{k+1}|  |a_{n-k+1}|}\\&< \dfrac{\sum_{n=1}^{\infty}\bigg[\left(\frac{21}{4}\right)^{\tfrac{n}{3}} \sum_{k=0}^n  |a_{k+1}|  |a_{n-k+1}|\bigg]}{1- \sum_{n=1}^{\infty} \sum_{k=0}^n |a_{k+1}|  |a_{n-k+1}|}.
\end{align*}
Since the denominator is positive, the follwoing inequality holds
\begin{align*}
    0<\dfrac{\sum_{n=1}^{\infty}\bigg[\left(\frac{21}{4}\right)^{\tfrac{n}{3}} \sum_{k=0}^n  |a_{k+1}|  |a_{n-k+1}|\bigg]}{1- \sum_{n=1}^{\infty} \sum_{k=0}^n |a_{k+1}|  |a_{n-k+1}|}< c, 
\end{align*}
provided \eqref{eqn:hypo-thm1} holds. This completes the proof of the result. 
\end{proof}

\section{Lemniscate starlikeness of Special functions}\label{sec:Appl-special}
Now, we are going to implement \Cref{thm:anlytic-LS} for constructing several analytic functions that are lemniscate starlike. 

\subsection{Example involving exponential function}\label{sec:Appl-exponential}
\begin{example}
    For our first example, consider the function
    \begin{align*}
        \mathtt{F_1}(z): = -\frac{3 (2 - 2 e^z + 2 z + z^2)}{z^2}=\sum _{n=1}^{\infty } \frac{z^n}{(4)_{n-1}}. 
    \end{align*}
    Clearly, $a_n= 1/(4)_{n-1}$ and a numerical calculation yields 
    \begin{align}
        \sum _{n=1}^{\infty } \left(\sum _{k=0}^n \frac{1}{(4)_k (4)_{n-k}}\right)=4 \left(56-45 e+9 e^2\right) = 0.71529<1.    
        \end{align}
Further, denote that 
\begin{align*}
    \sum _{n=1}^{\infty } \left(\left(\frac{21}{4}\right)^{n/3}-c\right) \sum _{k=0}^n \frac{1}{(4)_k (4)_{n-k}}=\Phi(c).
\end{align*}
The corresponding sum and convergence of the series give 
\[\Phi(c)=\frac{1}{49} \left(-196 (3 e-8) (3 e-7) c-16 e^{\frac{\sqrt[3]{21}}{2^{2/3}}} \left(8+4 \sqrt[3]{42}+42^{2/3}\right)+64 e^{\sqrt[3]{42}}+106 \sqrt[3]{42}+32\ 42^{2/3}+351\right). \]
A routine calculation implies $\Phi(c)- c$ is decreasing function of $c$ and finally, we conclude that 
 $z\mathtt{F_1}'(z)/\mathtt{F_1}(z) \prec \sqrt{1+cz}$ for $c \in [c_0, 1) $, where $c_0 \approx 0.990879$ is the solution of 
\[\frac{1}{49} \left(-196 (3 e-8) (3 e-7) c-16 e^{\frac{\sqrt[3]{21}}{2^{2/3}}} \left(8+4 \sqrt[3]{42}+42^{2/3}\right)+64 e^{\sqrt[3]{42}}+106 \sqrt[3]{42}+32\ 42^{2/3}+351\right)=c\] in $(0, 1]$. 
\end{example}

\subsection{Example involving confluent hypergeometric functions}\label{sec:Appl-conf-hyp}

A particularly important case is obtained by setting $p=0$ and $q=1$, leading to the \textit{confluent hypergeometric limit function}
\begin{equation}
{}_0F_1(;b;z) \;=\; \sum_{n=0}^{\infty} \frac{1}{(b)_n} \frac{z^n}{n!}.
\end{equation}
This function is sometimes denoted by $\Phi(b;z)$ in the classical literature.

The function ${}_0F_1(;b;z)$ is the unique analytic solution normalized by ${}_0F_1(;b;0)=1$ of the second-order linear differential equation
\begin{equation}
z\,y''(z) + b\,y'(z) - y(z) = 0.
\end{equation}

Since the parameters satisfy $p=0 < q+1=2$, the defining series converges for all $z \in \mathbb{C}$, and hence ${}_0F_1(;b;z)$ is an \textit{entire function}.

One of the most remarkable properties of ${}_0F_1$ is its close relation to Bessel functions. For $\nu \in \mathbb{C}$ we have
\begin{equation}
{}_0F_1(; \nu+1; z) = \Gamma(\nu+1)\, z^{-\nu/2}\, I_\nu(2\sqrt{z}),
\end{equation}
where $I_\nu(z)$ denotes the modified Bessel function of the first kind. Similarly,
\begin{equation}
{}_0F_1(; \nu+1; -z) = \Gamma(\nu+1)\, z^{-\nu/2}\, J_\nu(2\sqrt{z}),
\end{equation}
with $J_\nu(z)$ being the Bessel function of the first kind. 
Thus, ${}_0F_1$ provides a hypergeometric framework for Bessel functions and, more generally, for many special functions arising in analysis.

Some special evaluations of ${}_0F_1$ include:
\begin{align*}
{}_0F_1(;1;z) = \sum_{n=0}^\infty \frac{z^n}{(n!)^2} = I_0(2\sqrt{z}), \quad
{}_0F_1\!\left(; \tfrac{1}{2}; -\tfrac{x^2}{4}\right) = \cos(x), \quad
{}_0F_1\!\left(; \tfrac{3}{2}; -\tfrac{x^2}{4}\right) = \frac{\sin(x)}{x}.
\end{align*}
This illustrates how ${}_0F_1$ interpolates between trigonometric and Bessel functions.

Now consider our next example as $\mathtt{F_2}(z):=z {}_0F_1( \quad; b; z)$. Then, the corresponding sequence $a_n= 1/ ((n-1)! (b)_{n-1})$, and   from Lemma \ref{lemma-basic2} we have 
for $b>0$
\begin{align}\label{eqn-example2-1}
       \delta(b):= \sum _{n=1}^{\infty } \left(\sum _{k=0}^n \frac{1}{(b)_k k! (b)_{n-k} (n-k)!}\right) =  {}_1F_2\left(b-\frac{1}{2};b,2 b-1;4\right)-1
\end{align}
Using Mathematica software, we can numerically compute the root of the equation 
\begin{align}
{}_1F_2\left(b-\frac{1}{2};b,2 b-1;4\right)-1=1  
\end{align}
is $b=2.75885$. Further, from the \Cref{fig:example-2} it is clear that $\delta(b)<1$. 
\begin{figure}[h]
    \centering
    \includegraphics[width=0.5\linewidth]{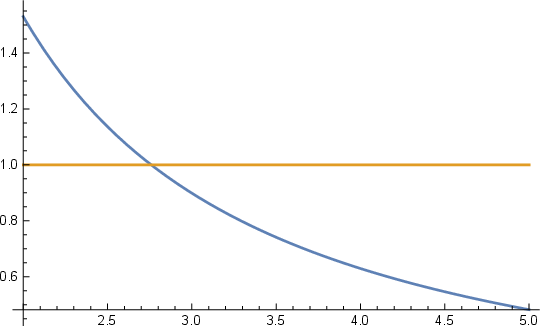}
    \caption{Mapping of $\delta(b)$ for $b>0$.}
    \label{fig:example-2}
\end{figure}

Hence $\delta(b)<1$ for $b>2.75885$. This leads to the following result:
\begin{theorem}\label{thm:example2}
For $b>2.75885$, $c \in (0, 1]$, if 
\begin{align}
  \, _1F_2\left(b-\frac{1}{2};b,2 b-1;2 \sqrt[3]{42}\right)  - c \, _1F_2\left(b-\frac{1}{2};b,2 b-1;4\right) <1,
\end{align}
then 
$z \mathtt{F_2}'(z)/\mathtt{F_2}(z) \prec \sqrt{1+c z}. $

In particular, $\mathtt{F_2}(z):=z {}_0F_1( \quad; b; z)$ is lemmniscate starlike for $b> 3.11423$. 
\end{theorem}
The proof the \Cref{thm:example2} can be done easily from \Cref{thm:anlytic-LS} due to the fact that 
\begin{align*}
    &\sum _{n=1}^{\infty } \left(\left(\tfrac{21}{4}\right)^{n/3}-c\right) \sum _{k=0}^n \frac{1}{\left(k! (b)_k\right) \left((n-k)! (b)_{n-k}\right)}\\ \quad &= \, _1F_2\left(b-\tfrac{1}{2};b,2 b-1;2 \sqrt[3]{42}\right)  - c \, _1F_2\left(b-\tfrac{1}{2};b,2 b-1;4\right)+c-1.
\end{align*}

 \subsection{Example involving Bessel functions}\label{sec:Appl-Bessel}
We recall some information about Bessel functions and their connections with geometric function theory. 
Bessel functions of the first kind, denoted by \( J_\nu(z) \), are solutions to Bessel’s differential equation:
\[
z^2 w''(z) + z w'(z) + (z^2 - \nu^2) w(z) = 0,
 \]
 where \(\nu\) is a real or complex parameter. These functions are central to many problems in mathematical physics, including heat conduction, wave propagation, and problems in circular and cylindrical domains. In the unit disc \(\mathbb{D} = \{ z \in \mathbb{C} : |z| < 1 \}\), \( J_\nu(z) \) is analytic because it has the power series expansion:
 \[
 J_\nu(z) = \sum_{n=0}^\infty \frac{(-1)^n}{n! \Gamma(n + \nu + 1)} \left( \frac{z}{2} \right)^{2n + \nu},
 \]
 which converges absolutely in the entire complex plane.

 For applications in geometric function theory, it is often convenient to consider the \emph{normalized form} of the Bessel function:
 \[
 f_3(\nu, z) := 2^\nu \Gamma(\nu + 1) z^{1 - \nu/2} J_\nu(\sqrt{z}),
 \]
 which satisfies the normalization conditions \( f_3(\nu, 0) = 0 \) and \( f_3'(\nu, 0) = 1 \). 
Similarly, the modified Bessel function of the first kind, \( I_\nu(z) \), is defined by:
\[
I_\nu(z) = \sum_{n=0}^\infty \frac{1}{n! \Gamma(n + \nu + 1)} \left( \frac{z}{2} \right)^{2n + \nu},
\]
and has a normalized form analogous to that of \( J_\nu(z) \) as 
\[
 f_4(\nu, z) = 2^\nu \Gamma(\nu + 1) z^{1 - \nu/2} I_\nu(\sqrt{z}).
\]

 In the study of the geometric properties of the function\(  f_3(\nu, z)\), several contributions have been made in the literature \cite{Selinger,Baricz-Ponnusamy-2010,Szasz-Kupan}. Selinger \cite{Selinger} demonstrated that the function \( f_{\nu}(z) \) is starlike in the unit disk \( \mathbb{D} \) for all \( \nu \geq 0.75 \). This range was subsequently improved by Baricz and Ponnusamy \cite{Baricz-Ponnusamy-2010}, who showed that starlikeness holds for \( \nu \geq -0.432619635\ldots \) or within the interval \( -0.875 \leq \nu < -0.719669914 \). It was also noted that the constant \( -0.875 \) is optimal in the sense that it is the smallest value of \( \nu \) for which \( f_{\nu} \) remains starlike in \( D \). A variety of techniques were utilized in \cite{Baricz-Ponnusamy-2010} to verify the starlikeness of \(  f_3(\nu, z) \) for \( \nu \geq -0.875 \), although a gap remains where the result has not been fully validated. By synthesizing all the results in \cite{Baricz-Ponnusamy-2010} concerning both \(  f_3(\nu, z) \) and \(  f_4(\nu, z) \), it can be concluded that 

\begin{theorem}\cite{Baricz-Ponnusamy-2010}\label{thmA}
For \( \nu \in [-0.875, -0.719669914) \cup (-0.4373468657, \infty) \) the functions \(  f_3(\nu, z) \) and \(  f_4(\nu, z) \) are starlike in $\mathbb{D}$. 
\end{theorem}
Our next consideration is the generalized Bessel function, introduced in \cite{Baricz1,Baricz3}.  The generalized and normalized Bessel functions have the series form
	\begin{align*}
		\mathtt{U}_{\nu}(z) := \sum_{n=0}^\infty \frac{(-1)^n \eta^n}{4^n (\kappa)_n} \; \frac{z^n}{n!}, \quad  2 \kappa = 2 p+ b+1 \neq 0, -2, -4, -6, \ldots;
	\end{align*}
	which is the solution of
	\begin{align}\label{Bessel-gen-DE}
		4 z^2 \mathtt{U''}(z)+ 4 \kappa z \mathtt{U'}(z)+ \eta z \mathtt{U}(z)=0.
	\end{align}
	{For} $b=c=1$, the~function $\mathtt{U}_{\nu}$ represents the normalized Bessel function of order $p$, while for  $b=- \eta =1$, the~function $\mathtt{U}_{\nu}$ represents the normalized modified Bessel function of order $p$. The~Spherical Bessel function can also be obtain by using $b=2$, $\eta=1$.
	
	The inclusion of $\mathtt{U}_{\nu}$ in various subclasses of univalent functions theory is extensively studied  by many authors~\cite{Kanas-saiful,Madan-Kumar-Ravi-Hungar,saiful-Dhuian,saiful-swami,Baricz3,Baricz-Pswamy,Baricz-Szasz,Baricz1,Baricz2} and some references therein. Recently, the~lemniscate starlike, convexity and other properties of $\mathtt{U}_{\nu}$ is studied in~\cite{Madan-Kumar-Ravi-Hungar,Lemniscate,Zayed-Lem-Bessel}.

   The function $\mathtt{U}_{\nu}$ specializes to the normalized Bessel function of order $p$ when $b=\eta=1$, and to the normalized modified Bessel function of order $p$ when $b=-\eta=1$. By selecting $b=2$ and $c=1$, one also obtains the spherical Bessel function.    

	Let us consider
	\begin{align*}
		\mathtt{F_3}(\kappa, \eta, z):= z \mathtt{U}_p(z).
	\end{align*}
	It can be shown that 
	\begin{align}
		4 z f_4''(z)+ 4(\kappa-2) z f_4'(z)+ (cz- 4 \kappa+8)f_4(z)=0.
	\end{align}

Now we have the following result regarding the lemniscate starlike property of generalised Bessel functions: 
\begin{theorem}\label{thm-GBessel}
    For real $\kappa>0$ and $\eta \in \mathbb{C}$, the function $\mathtt{F_3}$ is lemniscate starlike if for $c \in (0, 1]$, we have
    \begin{align}\label{en:hypo-1-thm-Bessel}
       \, _1F_2\left(\kappa -\frac{1}{2};\kappa ,2 \kappa -1;\frac{|\eta|\sqrt[3]{21}}{2^{2/3}}\right)-1<c\, _1F_2\left(\kappa -\frac{1}{2};\kappa ,2 \kappa -1; |\eta|\right) \le 2 c.
    \end{align} 
    Then, the function $\mathtt{F_3}(\kappa, \eta, z)$ is lemniscate starlike. 
\end{theorem}

\begin{proof}
   If we denote 
\[
   a_k= \frac{(-\eta)^{k-1}}{4^{\,k-1}(k-1)!\, (\kappa)_{k-1}}, 
\]
then 
\begin{align*}
		\mathtt{F_3}(\kappa, \eta, z):= z \mathtt{U}_p(z)= z+\sum_{k=2}^{\infty} a_k z^{k}.
\end{align*}
In order to prove the result by applying \Cref{thm:anlytic-LS}, it suffices to verify the hypotheses given in \eqref{eqn:hypo-1-thm1} and \eqref{eqn:hypo-thm1}, namely 
\begin{align}\label{eqn:hypo-1-thm1-appl_B}
\sum_{n=1}^{\infty}  \sum_{k=0}^n  \frac{|\eta|^{k}}{4^k k! |(\kappa)_{k}|}  \frac{|\eta|^{\,n-k}}{4^{\,n-k} (n-k)! |(\kappa)_{n-k}|} < 1,
\end{align}
and 
\begin{align}\label{eqn:hypo-thm1-appl_B}
\sum_{n=1}^{\infty} \left( \left(\tfrac{21}{4}\right)^{\tfrac{n}{3}}- c\right) \sum_{k=0}^n  \frac{|\eta|^{k}}{4^k k! |(\kappa)_{k}|}  \frac{|\eta|^{\,n-k}}{4^{\,n-k} (n-k)! |(\kappa)_{n-k}|} < c.
\end{align}

\noindent
Consider first the sum on the left-hand side of \eqref{eqn:hypo-1-thm1-appl_B}.  
Since $\kappa>0$, it follows that $|(\kappa)_{k}|=(\kappa)_{k}$.  
Applying \Cref{lem:identity-1}, we obtain
\begin{align*}
\sum_{n=1}^{\infty} \frac{|\eta|^{n}}{4^n} \sum_{k=0}^n  
 \frac{1}{ k! (\kappa)_{k}\, (n-k)! (\kappa)_{n-k}} 
&=\sum_{n=1}^{\infty} \frac{|\eta|^{n}}{4^n}  
   \frac{ 2^{2n} \left(\kappa-\tfrac{1}{2}\right)_{n}}{\left(2\kappa-1\right)_{n} \left(\kappa\right)_{n} n!}\\
&=  {}_1F_{2}\!\left( \kappa-\tfrac{1}{2};\, 2\kappa-1,\, \kappa;\,  |\eta|\right)-1.
\end{align*}
Therefore, the inequality \eqref{eqn:hypo-1-thm1-appl_B} holds provided that 
\[
 {}_1F_{2}\!\left( \kappa-\tfrac{1}{2};\, 2\kappa-1,\, \kappa;\,  |\eta|\right) < 2.
\]

\noindent
Next, consider the sum in \eqref{eqn:hypo-thm1-appl_B}.  
Proceeding as before, we have
\begin{align*}
&\sum_{n=1}^{\infty} \left( \left(\tfrac{21}{4}\right)^{\tfrac{n}{3}}- c\right) 
  \frac{|\eta|^{n}}{4^n} \sum_{k=0}^n  
   \frac{1}{4^k k! (\kappa)_{k}\, (n-k)! (\kappa)_{n-k}} \\
&=\sum_{n=1}^{\infty} \left( \left(\tfrac{21}{4}\right)^{\tfrac{n}{3}}- c\right) 
   \frac{|\eta|^{n}}{4^n}  
   \frac{ 2^{2n} \left(\kappa-\tfrac{1}{2}\right)_{n}}
   {\left(2\kappa-1\right)_{n} \left(\kappa\right)_{n} n!}\\
&= {}_1F_{2}\!\left( \kappa-\tfrac{1}{2};\, 2\kappa-1,\, \kappa;\, 
   \left(\tfrac{21}{4}\right)^{\tfrac{1}{3}} |\eta| \right)- 1
   - c\; {}_1F_{2}\!\left( \kappa-\tfrac{1}{2};\, 2\kappa-1,\, \kappa;\,  |\eta|\right)+c.
\end{align*}
Hence, the inequality \eqref{eqn:hypo-thm1-appl_B} holds if
\begin{align*}
 {}_1F_{2}\!\left( \kappa-\tfrac{1}{2};\, 2\kappa-1,\, \kappa;\, 
   \left(\tfrac{21}{4}\right)^{\tfrac{1}{3}} |\eta| \right)- 1
 < c\; {}_1F_{2}\!\left( \kappa-\tfrac{1}{2};\, 2\kappa-1,\, \kappa;\,  |\eta|\right).
\end{align*}
\noindent
Thus, both hypotheses of \Cref{thm:anlytic-LS} are satisfied for $c \in (0, 1]$ whenever condition \eqref{en:hypo-1-thm-Bessel} holds. Consequently, the desired conclusion follows.
\end{proof}
By taking $\eta=\pm 1$ and $c=1$, numerically and graphically (See \Cref{fig:thm-GBessel}), it can be shown that the hypothesis of \Cref{thm-GBessel} holds for $\kappa >0.694651$. 

\begin{figure}[h]
    \centering
    \includegraphics[width=0.6\linewidth]{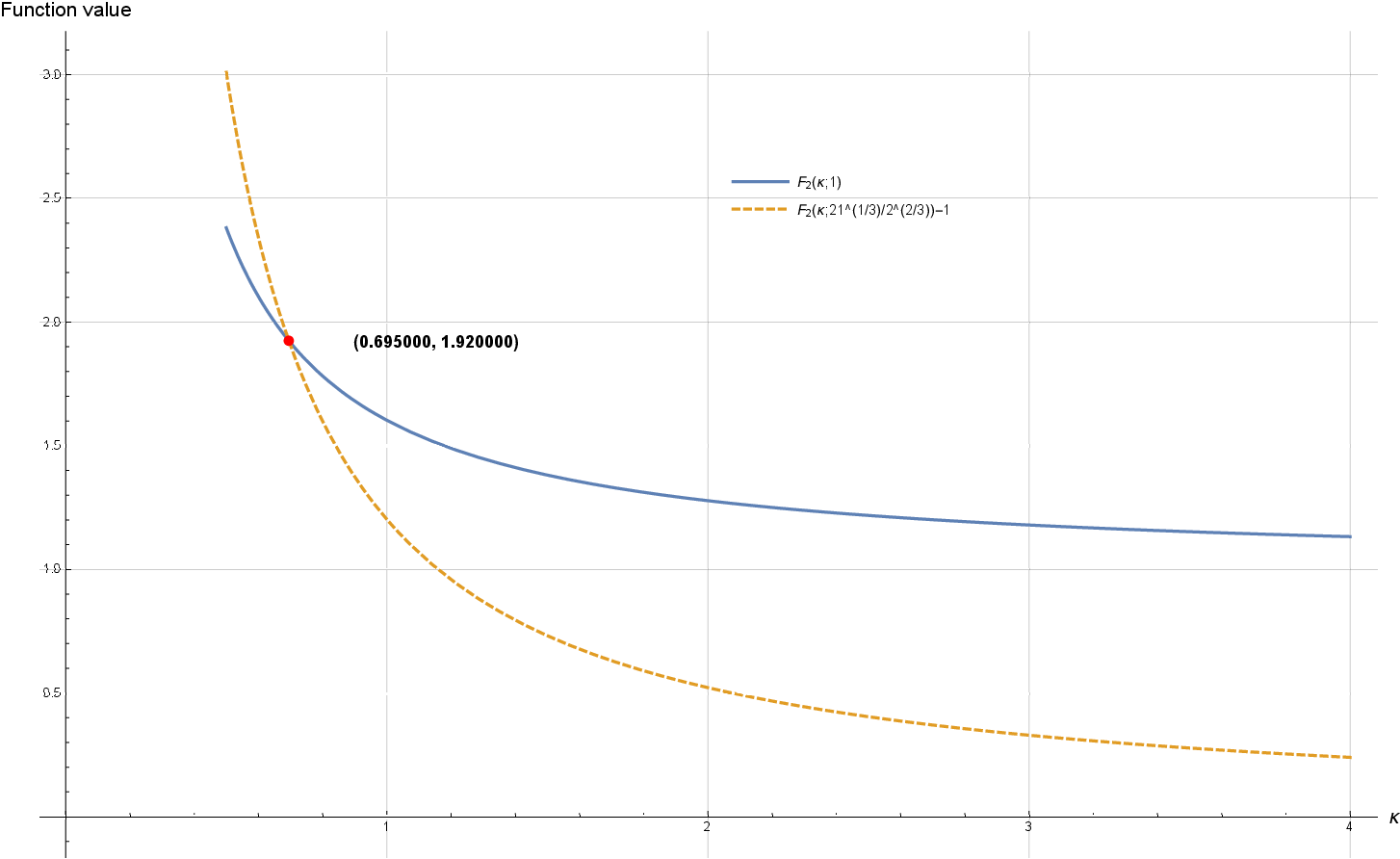}
    \caption{Special case of \Cref{thm-GBessel}}
    \label{fig:thm-GBessel}
\end{figure}

Since for $\eta= \pm 1$ and $\kappa= \nu+1$, the functions $\mathtt{F_3}(\kappa, \eta, z):= z \mathtt{U}_{\nu}(z)$ is reduces to the normalized form of classical Bessel $f_3(\nu, z)$ and modified Bessel $f_4(\nu, z)$ functions respectively, by taking $c=1$, we have the following special result. 
\begin{corollary}
    The normalized Bessel $f_3(\nu, z)$ and modified Bessel $f_4(\nu, z)$ functions of order $\nu$ are lemniscate starlike for $\nu > -0.305349$. 
\end{corollary}

\begin{remark}
   We can prove \Cref{thm-GBessel} using an alternative method. It is well-known that 
   \begin{align}\label{eqn:self-CP}
       \sum_{n=0}^{\infty} \left( \sum_{k=0}^n d_k d_{n-k} \right) z^n = \left(\sum_{n=0}^{\infty}  d_n  z^n \right)^2. 
   \end{align}
Using this identity, \eqref{eqn:hypo-1-thm1-appl_B} and  \eqref{eqn:hypo-thm1-appl_B} is equivalent to 
   \begin{align}\label{eqn:self-CP-B1}
\sum_{n=1}^{\infty}  \sum_{k=0}^n  \frac{|\eta|^{k}}{4^k k! |(\kappa)_{k}|}  \frac{|\eta|^{\,n-k}}{4^{\,n-k} (n-k)! |(\kappa)_{n-k}|} =\left(\sum_{n=0}^{\infty} \frac{|\eta|^{n}}{4^n n! (\kappa)_{n}}\right)^2 -1={\left(\, _0F_1\left(;\kappa ;\frac{|\eta|}{4}\right)\right)}^2 -1  = < 1,
\end{align}
and 
\begin{align}\label{eqn:self-CP-B2} \notag
&\sum_{n=1}^{\infty} \left( \left(\tfrac{21}{4}\right)^{\tfrac{n}{3}}- c\right) \sum_{k=0}^n  \frac{|\eta|^{k}}{4^k k! |(\kappa)_{k}|}  \frac{|\eta|^{\,n-k}}{4^{\,n-k} (n-k)! |(\kappa)_{n-k}|}\\&=\left({\left(\, _0F_1\left(;\kappa ;\frac{|\eta| 21^{\tfrac{1}{3}}}{4^{\tfrac{4}{3}}}\right)\right)}^2 -1\right)- c \left( {\left(\, _0F_1\left(;\kappa ;\frac{|\eta|}{4}\right)\right)}^2 -1\right)<c. 
\end{align}
Now, by taking $|\eta|=1$ and $c=1$, numerical verification confirms that the inequalities \eqref{eqn:self-CP-B1} and \eqref{eqn:self-CP-B2} hold simultaneously only for $\kappa >0.694651$, which is consistent with the threshold established earlier. 
\end{remark}

 Regarding the lemniscate starlikeness of Bessel functions, the best results available in the literature are given in \cite{Zayed-Bulboacua}. For our best interest for the comparison, we state the complete theorem as follows

\begin{theorem}\label{thm-Zayed-LS}\cite{Zayed-Bulboacua}
    Assume $\eta \in \mathbb{C}\setminus \{0\}$ and $\kappa \in \mathbb{C}\setminus \mathbb{Z}_0^{+}$ with $\real(\kappa)>0$. 
    \begin{align*}
        0 < \frac{e^{\left|\tfrac{\eta}{\kappa}\right|}-1}{2-e^{\left|\tfrac{\eta}{2\kappa}\right|}}< c,
    \end{align*}
implies $ \mathtt{F_3}(\kappa, \eta, z) \in \mathtt{S}^{\ast}(q_c). $
\end{theorem}

\begin{remark}\label{remark-1}
 Although \Cref{thm-Zayed-LS} is formulated for complex values of $\kappa$, it is instructive to compare \Cref{thm-Zayed-LS} with \Cref{thm-GBessel} for a fixed $\eta \in \mathbb{C} \setminus \{0\}$ and real $\kappa$. 
For $\eta = \pm 1$ or $\eta = \pm i$, \Cref{thm-GBessel} ensures that $\mathtt{F_3}(\kappa, \eta, z) \in \mathtt{S}^{\ast}(q_1)$ whenever $\kappa > 0.694651$, whereas the same conclusion follows from \Cref{thm-Zayed-LS} only when $\kappa > 1.6459$. 
Hence, \Cref{thm-GBessel} provides a better admissible range for $\kappa$.
\end{remark}

Next, to present a complete comparison of our results with \Cref{thm-Zayed-LS}, we now state an analogue of \Cref{thm-GBessel} that is valid for $\eta \in \mathbb{C}\setminus\{0\}$ and $\kappa \in \mathbb{C}\setminus\mathbb{Z}_0^{+}$. 
In this setting, we make use of the inequality $|(\kappa)_k| \geq |\kappa|^{k}$ and deduce that the condition \eqref{eqn:hypo-1-thm1-appl_B} is equivalent to
\[
e^{\frac{|\eta|}{2|\kappa|}} < 2,
\]
which clearly holds whenever $|\kappa| > \dfrac{|\eta|}{2\ln(2)}$. 
This equivalence follows from the application of \Cref{lem:identity-2} to the series in \eqref{eqn:hypo-1-thm1-appl_B}, yielding
\begin{align*}
\sum_{n=1}^{\infty}\sum_{k=0}^n 
\frac{|\eta|^{k}}{4^{k} k! |(\kappa)_{k}|}
\frac{|\eta|^{\,n-k}}{4^{\,n-k} (n-k)! |(\kappa)_{n-k}|}
&< 
\sum_{n=1}^{\infty} \frac{|\eta|^{n} 2^n}{4^n |\kappa|^{n} n!}
= e^{\frac{|\eta|}{2|\kappa|}} - 1.
\end{align*}

Similarly, the inequality \eqref{eqn:hypo-thm1-appl_B} reduces to
\begin{align*}
\sum_{n=1}^{\infty}
\left( \left(\tfrac{21}{4}\right)^{\tfrac{n}{3}} - c \right)
\sum_{k=0}^n 
\frac{|\eta|^{k}}{4^{k} k! |(\kappa)_{k}|}
\frac{|\eta|^{\,n-k}}{4^{\,n-k} (n-k)! |(\kappa)_{n-k}|}
&< 
-c\left( e^{\frac{|\eta|}{2|\kappa|}} - 1 \right)
+ e^{\frac{\sqrt[3]{21}|\eta|}{2^{5/3}|\kappa|}} - 1 < c,
\end{align*}
which further implies
\[
e^{\frac{\sqrt[3]{21}|\eta|}{2^{5/3}|\kappa|}} - c\, e^{\frac{|\eta|}{2|\kappa|}} < 1.
\]
We now formulate the corresponding theorem without detailed proof.

\begin{theorem}\label{thm:thm1-appl_B-complex}
Let $\eta \in \mathbb{C}\setminus\{0\}$ and $\kappa \in \mathbb{C}\setminus\mathbb{Z}_0^{+}$ satisfy $|\kappa| > \dfrac{|\eta|}{2\ln(2)}$. 
If
\begin{align}\label{eqn:hypo-1-thm1-appl_B-complex}
e^{\frac{\sqrt[3]{21}|\eta|}{2^{5/3}|\kappa|}} - c\, e^{\frac{|\eta|}{2|\kappa|}}< 1, 
\end{align}
then $\mathtt{F_3}(\kappa, \eta, z) \in \mathtt{S}^{\ast}(q_c).$
\end{theorem}

\begin{remark}\label{remark-2}
We emphasise here the independent significance of \Cref{thm-GBessel} and \Cref{thm:thm1-appl_B-complex}. 

\begin{enumerate}
    \item Note that \Cref{thm:thm1-appl_B-complex} also remains valid for real $\kappa$. 
By revisiting the case discussed in \Cref{remark-1}, it can be shown that $\mathtt{F_3}(\kappa, \eta, z) \in \mathtt{S}^{\ast}(q_1)$ whenever $\kappa > 0.840149$. 
Therefore, for real $\kappa > 0$, \Cref{thm-GBessel} provides a comparatively better bound than \Cref{thm:thm1-appl_B-complex}.
\item 
Next, consider the case when $\kappa = \tfrac{3}{2} + i a$, where $a \in \mathbb{R}$, and $\eta \in \mathbb{C}$ such that $|\eta| = 1$. Clearly, for all $a \in \mathbb{R}$, we have
\[
|\kappa| = \sqrt{\frac{9}{4} + a^2} > \frac{1}{2 \ln(2)}.
\]
Hence, the right-hand side of \eqref{eqn:hypo-1-thm1-appl_B-complex} can be rewrite as
\begin{align}\label{eq:h-of-a}
    \mathfrak{h}(a) = e^{\frac{\sqrt[3]{21}}{ 2^{5/3} \sqrt{a^2 + \tfrac{9}{4}}}} - e^{\frac{1}{2 \sqrt{a^2 + \tfrac{9}{4}}}}.
\end{align}
By elementary analysis, it can be shown that $\mathfrak{h}(a)$ attains its maximum at $a = 0$. Numerically,
\begin{align}
    \mathfrak{h}(0) 
    &= e^{\frac{\sqrt[3]{21}}{ 2^{5/3} \times \tfrac{3}{2}}} - e^{\frac{1}{2 \times \tfrac{3}{2}}} 
    = e^{\frac{\sqrt[3]{21}}{ 3 \times 2^{2/3}}} - e^{\tfrac{1}{3}} 
    = 0.389244.
\end{align}
Therefore, for this specific choice of $\kappa$ and $\eta$, \Cref{thm:thm1-appl_B-complex} ensures that 
\[
\mathtt{F_3}\big(\tfrac{3}{2} + i a, \eta, z\big) \in \mathtt{S}^{\ast}(q_1) \quad \text{for all } a \in \mathbb{R}.
\]

On the other hand, consider
\begin{align}\label{eq:k-of-a}
    \mathfrak{k}(a) = \frac{e^{\tfrac{1}{\sqrt{a^2 + \tfrac{9}{4}}}} - 1}{2 - e^{\tfrac{1}{2\sqrt{a^2 + \tfrac{9}{4}}}}}.
\end{align}
A straightforward computation shows that $\mathfrak{k}(a)$ also achieves its maximum at $a = 0$, and 
\begin{align}
    \mathfrak{k}(0) = \frac{e^{\tfrac{2}{3}} - 1}{2 - e^{\tfrac{1}{3}}} = 1.56809.
\end{align}
Moreover, $\mathfrak{k}(a) < 1$ whenever $|a| > 1.15045$. Consequently, by \Cref{thm-Zayed-LS}, it follows that
\[
\mathtt{F_3}\big(\tfrac{3}{2} + i a, \eta, z\big) \in \mathtt{S}^{\ast}(q_1)
\quad \text{for all } a \in \mathbb{R} \text{ with } |a| > 1.15045.
\]
Hence, it can be concluded that \Cref{thm:thm1-appl_B-complex} provides a significantly stronger result than \Cref{thm-Zayed-LS}.
\end{enumerate}
\end{remark}

\subsection{Example involving  cross-product of Bessel functions}\label{sec:Appl-CP-Bessel}

In this part we are considering  cross-product of Bessel and modified Bessel functions, which have a power series \cite{Al-Kharsani-Bessel-prod}  representation  
If $\nu > -1$ and $z \in \mathbb{C}$, then the following power series representation holds:
\begin{align}\label{eqn:corssproduct-Bessel}
  J_{\nu+1}(z) I_\nu(z) + J_\nu(z) I_{\nu+1}(z) = 2 \sum_{n \geq 0} \frac{(-1)^n \left( \frac{z}{2} \right)^{2\nu + 4n + 1}}{n! \, \Gamma(\nu + n + 1)\Gamma(\nu + 2n + 2)}.  
\end{align}

We consider the normalised form of the cross-product \eqref{eqn:corssproduct-Bessel} defined as 
\begin{align}\label{eqn:corssproduct-Bessel-normal}\notag
 \mathcal{CP}_{B}(\nu, z)&= 2^{2\nu} z^{-\tfrac{\nu}{2}+\tfrac{3}{4}} \Gamma(\nu + 1) \Gamma(\nu + 2) \left( J_{\nu+1}(\sqrt[4]{z}) I_\nu(\sqrt[4]{z}) + J_\nu(\sqrt[4]{z}) I_{\nu+1}(\sqrt[4]{z})\right) \\\notag
 &  = \sum_{n \geq 0} \frac{(-1)^n \Gamma(\nu + 1) \Gamma(\nu + 2) z^{n+1}}{n! \, \Gamma(\nu + n + 1)\Gamma(\nu + 2n + 2) 16^n}\\
 &= z+ \sum_{n \geq 2} \frac{(-1)^{n-1} \Gamma(\nu + 1) \Gamma(\nu + 2) z^{n}}{(n-1)! \, \Gamma(\nu + n )\Gamma(\nu + 2n ) 16^{n-1}}. 
\end{align}
Now, recall well-known identities involving Pochhammer symbol and gamma functions: 
\[\Gamma(\nu+n)= (\nu)_n\Gamma(\nu)=(\nu+1)_{n-1}\Gamma(\nu+1)\] 
and 
\begin{align}
\Gamma(\nu+2n)= (\nu)_{2n}\Gamma(\nu) =   2^{2n} \left(\tfrac{\nu}{2}\right)_n \left(\tfrac{\nu+1}{2}\right)_n \Gamma(\nu)=2^{2n-2} \left(\tfrac{\nu+2}{2}\right)_{n-1} \left(\tfrac{\nu+3}{2}\right)_{n-1} \Gamma(\nu+2). 
\end{align}
Using this, we can rewrite \eqref{eqn:corssproduct-Bessel-normal} as 
\begin{align}\label{eqn:corssproduct-Bessel-normal1}
 \mathcal{CP}_{B}(\nu, z)
 &= z+ \sum_{n \geq 2} \frac{(-1)^{n-1}  z^{n}}{(n-1)!\, \left(\nu+1\right)_{n-1} \, \left(\frac{\nu+2}{2}\right)_{n-1} \left(\frac{\nu+3}{2}\right)_{n-1} 64^{n-1}}. 
\end{align}
The geometric properties of $\mathcal{CP}_{B}(\nu, z)$ is studied in \cite{Al-Kharsani-Bessel-prod,Baricz-product-Bessel,Baricz-CP-Bessel}. It is shown in \cite{Al-Kharsani-Bessel-prod} that $\mathcal{CP}_{B}(\nu, z)$ is starlike in $\mathbb{D}$ and all of its derivatives are close-to-convex (and hence univalent) there if and only if $\nu \ge \nu_\ast$, where $\nu_\ast \approx 0.94$ is the unique root of the equation on $(-1,\infty)$
\[
(\nu - 1) J_{\nu}(1) I_{\nu+1}(1)
+ (\nu - 1) J_{\nu+1}(1) I_{\nu}(1)
= J_{\nu}(1) I_{\nu}(1).
\]
In  \cite{Baricz-product-Bessel,Baricz-CP-Bessel}, radius of starlikeness and convexity of three different normalized forms of $\mathcal{CP}_{B}(\nu, z)$ have been studied by using infinite factorization involving real zeroes of $\mathcal{CP}_{B}(\nu, z)$. Following the notation of \cite{Baricz-product-Bessel}, one of the results in \cite{Baricz-product-Bessel} can be stated as: 
 The function $h_{\nu}=\mathcal{CP}_{B}(\nu, .)$ is convex of order $\alpha$ in $D$ if and only if 
\[
\nu \geq \nu^{c}_{\alpha}(h_{\nu}),
\]
where $\nu^{c}_{\alpha}(h_{\nu})$ is the unique root of the equation 
\[
h_{\nu}''(1) + (1 - \alpha) h_{\nu}'(1) = 0.
\]   
In particular, $\alpha=0$ we have $h_{\nu}$ is convex for $\nu \geq \nu^{c}_{0}  (h_{\nu}) \approx -0.893305$, where $\nu^{c}_{0}  (h_{\nu})$ is the root of  $h_{\nu}''(1) +  h_{\nu}'(1) = 0$ for $\nu>-1$. 

In \cite{Baricz-CP-Bessel}, it is shown that the function $h_{\nu}$ is starlike of order $\alpha$ in $D$ if and only if 
$
\nu \geq \nu_{\alpha}^{*}(h_{\nu}),
$
where $\nu_{\alpha}^{*}(h_{\nu})$ is the unique root of the equation
\begin{align} \label{eq:star_condition}
J_{\nu}(1) I_{\nu}(1) - \big( 2\alpha + \nu - 1 \big) \big( J_{\nu+1}(1) I_{\nu}(1) + J_{\nu}(1) I_{\nu+1}(1) \big) = 0.
\end{align}
In particular, $h_{\nu}$ is starlike in $D$ if and only if 
$
\nu \geq \nu_{0}^{*}(h_{\nu}) \simeq -0.94\ldots,
$
where $\nu_{0}^{*}(h_{\nu})$ is the unique root of equation~\eqref{eq:star_condition} when $\alpha = 0$.

Next,  we state our result for lemniscate starlike of $\mathcal{CP}_{B}(\nu, z)$  for real $\nu>-1$ by applying \eqref{eqn:self-CP}. 
\begin{theorem}
    For $\nu \in \mathbb{R}$ with $\nu>-1$, if 
\begin{align}
   \left( {}_0F_{3}\left(- ; \nu+1, \tfrac{\nu+2}{2}, \tfrac{\nu+3}{2};  \tfrac{\sqrt[3]{21}}{64 \sqrt[3]{4} }\right)\right)^2 -1< c \left( {}_0F_{3}\left(- ; \nu+1, \tfrac{\nu+2}{2}, \tfrac{\nu+3}{2};  \tfrac{1}{64}\right)\right)^2< 2 c,
\end{align}
then $\mathcal{CP}_{B}(\nu, z) \in \mathcal{S}^{\ast}(q_c)$. In particular, $\mathcal{CP}_{B}(\nu, z) \in \mathcal{S}^{\ast}(q_1)$ for $\nu > \nu_0 = - 0.931564$.  
Here, $\nu_0$ is the root of  $\left( {}_0F_{3}\left(- ; \nu+1, \tfrac{\nu+2}{2}, \tfrac{\nu+3}{2};  \tfrac{1}{64}\right)\right)^2=2$.   
\end{theorem}
We omit the detailed proof. However, for the special case $c=1$, the range of $\nu$ is determined by the numerical calculation of the roots of
\begin{enumerate}
    \item[(i)] the identity
    \begin{align}
   \left( {}_0F_{3}\left(- ; \nu+1, \tfrac{\nu+2}{2}, \tfrac{\nu+3}{2};  \tfrac{\sqrt[3]{21}}{64 \sqrt[3]{4} }\right)\right)^2 -1=  \left( {}_0F_{3}\left(- ; \nu+1, \tfrac{\nu+2}{2}, \tfrac{\nu+3}{2};  \tfrac{1}{64}\right)\right)^2
\end{align}
which is $\nu_1= -0.933296$,
\item[(ii)] the identity 
\begin{align}
 \left( {}_0F_{3}\left(- ; \nu+1, \tfrac{\nu+2}{2}, \tfrac{\nu+3}{2};  \tfrac{1}{64}\right)\right)^2= 2 
\end{align}
which is $\nu_2= -0.931564$.
\end{enumerate}
Thus $\nu_0= \max\{\nu_1, \nu_2\}=\nu_2$. 

\begin{remark}
    Although a rigorous theoretical justification has not yet been established, graphical evidence strongly suggests that the best lower bound for $\nu$ is given by $\nu_0 \approx -0.931$. This observation is clearly supported by the behavior depicted in \Cref{fig:CPB-Bessel}.
\begin{figure}[h!]
    \centering
    \includegraphics[width=0.5\linewidth]{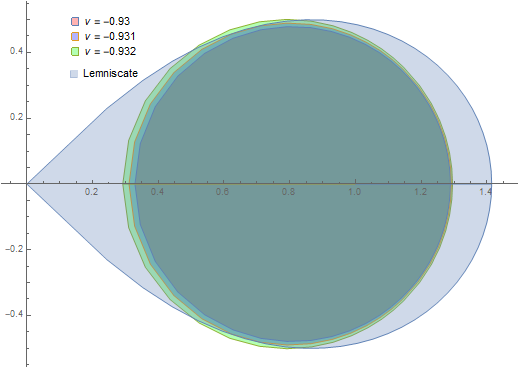}
    \caption{Best lower bound of $\nu$ such that $\mathcal{CP}_{B}(\nu, z) \in S^{\ast}(q_c)$.}
    \label{fig:CPB-Bessel}
\end{figure}

The graphical pattern indicates that for $\nu < \nu_0$, the corresponding mapping fails to remain entirely within the desired domain, whereas for $\nu \ge \nu_0$, the mapping lies completely inside. This provides compelling visual support for considering $\nu_0$ as the optimal threshold.

\end{remark}



\subsection{Example involving Struve functions}\label{sec:Appl-Struve}
Struve functions of the first kind, denoted by $S_\nu(z)$, are solutions to non-homegenous Bessel’s differential equation:
\begin{align} \label{eqn:SF-ODE}
z^2 w''(z) + z w'(z) + (z^2 - \nu^2) w(z) = \frac{4 z^{\nu+1}}{2^{\nu+1}\sqrt{\pi} \Gamma\left(\nu+1/2\right)},
\end{align}
where $\nu$ is a real or complex parameter. In the unit disk $\mathbb{D}$, $S_\nu(z)$ is analytic because it has the power series expansion
\[
S_\nu(z) = \sum_{n=0}^\infty \frac{(-1)^n}{\Gamma(n + 3/2) \,\Gamma(n+\nu+ 3/2)} \left( \frac{z}{2} \right)^{2n+\nu+1},
\]
which is entire. The modified Struve function is given as 
\[
L_\nu(z) = \sum_{n=0}^\infty \frac{1}{\Gamma(n + 3/2) \,\Gamma(n+\nu+ 3/2)} \left( \frac{z}{2} \right)^{2n+\nu+1}.
\]

Struve functions have gone through with several generalization \cite{Yagmur-Orhan, Ali-Mondal-Nisar}. The generalization given in \cite{Yagmur-Orhan} can be represented as a power series as 
\begin{align}\label{eqn:GSF-Orhan}
   \mathcal{W}_{\nu, b, \eta}(z)= \sum_{n=0}^{\infty} \frac{(-\eta)^n}{\Gamma(n + 3/2) \,\Gamma\left(n+\nu+ \tfrac{b+2}{2} \right)}\left( \frac{z}{2} \right)^{2n+\nu+1}.
\end{align}
The function $\mathcal{W}_{\nu, b, \eta}(z)$ is the solution of the differential equation 
\begin{align} \label{eqn:GSF-ODE}
z^2 w''(z) + b z w'(z) + (\eta z^2 - \nu^2+ (1-b) \nu) w(z) = \frac{4 z^{\nu+1}}{2^{\nu+1}\sqrt{\pi} \Gamma\left(\nu+b/2\right)}. 
\end{align}

For inclusion in geometric function theory, both $\mathcal{W}_{\nu, b, \eta}(z)$ and ${}_a\mathcal{U}_{\nu, b, \eta}(z)$ can be respectifully normalised as 
\begin{align}\label{eqn:GSF-Orhan-normal} 
   \mathtt{W}_{\nu, b, \eta}(z)&= 2^{\nu} \sqrt{\pi} z^{-\nu} \Gamma\left(\nu+ \tfrac{b+2}{2} \right) \mathcal{W}_{\nu, b, \eta}(\sqrt{z}) 
    =z+\sum_{n=2}^{\infty} \tfrac{(-\eta)^{n-1}}{\left(\tfrac{3}{2}\right)_{n-1} \,\left(\kappa_{s}  \right)_{n-1} 4^{n-1}} z^n,
\end{align}
and 
\begin{align}\label{eqn:GSF-a-normal} 
   {}_a\mathtt{U}_{\nu, b, \eta}(z)&= 2^{\nu} \sqrt{\pi} z^{-\nu} \Gamma\left(\nu+ \tfrac{b+2}{2} \right) \mathcal{U}_{\nu, b, \eta}(\sqrt{z})    
   =z+\sum_{n=2}^{\infty} \tfrac{(-\eta)^{n-1}}{\left(\tfrac{3}{2}\right)_{n-1} \,\left(\kappa_{s} \right)_{a(n-1)} 4^{n-1}}z^{n}
\end{align}
here $\kappa_{s}=\nu+ \tfrac{b+2}{2}.$ 

Applying identity \eqref{eqn:self-CP}, we have the following relation:
\begin{align*}
\sum_{n=1}^{\infty} \sum_{k=0}^n \tfrac{|\eta|^{k}}{\left(\tfrac{3}{2}\right)_{k} \,\left(\kappa_{s}  \right)_{k} 4^{k}} \tfrac{|\eta|^{n-k}}{\left(\tfrac{3}{2}\right)_{n-k} \,\left(\kappa_{s}  \right)_{n-k} 4^{n-k}}
&=\left(\sum_{n=0}^{\infty}  \tfrac{|\eta|^{n}}{\left(\tfrac{3}{2}\right)_{n} \,\left(\kappa_{s}  \right)_{n} 4^{n}} \right)^2-1\\
&=\left[\,{}_{1}F_{2}\!\left(1; \tfrac{3}{2},\ \kappa_{s};\;\frac{|\eta|}{4}\right)\right]^{2} - 1. 
\end{align*}
Similarly, 
\begin{align*}
\sum_{n=1}^{\infty} \left(\tfrac{21}{4}\right)^{n/3}\sum_{k=0}^n \tfrac{|\eta|^{k}}{\left(\tfrac{3}{2}\right)_{k} \,\left(\kappa_{s}  \right)_{k} 4^{k}} \tfrac{|\eta|^{n-k}}{\left(\tfrac{3}{2}\right)_{n-k} \,\left(\kappa_{s}  \right)_{n-k} 4^{n-k}}
&=\left[\,{}_{1}F_{2}\!\left(1; \tfrac{3}{2},\ \kappa_{s};\;\frac{\sqrt[3]{21}|\eta|}{4\; \sqrt[3]{4}}\right)\right]^{2} - 1. 
\end{align*}
Using \Cref{thm:anlytic-LS}, we have the following result: 
\begin{theorem}\label{thm:Lem-GStruve}
   Suppose that $\kappa_s>0$ and $\eta \in \mathbb{C}\setminus \{0\}$. If 
    \begin{align}\label{eqn:hypo-Lem-GStruve}
        \left[\,{}_{1}F_{2}\!\left(1; \tfrac{3}{2},\ \kappa_{s};\;\frac{\sqrt[3]{21}|\eta|}{4\; \sqrt[3]{4}}\right)\right]^{2} - 1< c \left[\,{}_{1}F_{2}\!\left(1; \tfrac{3}{2},\ \kappa_{s};\;\frac{|\eta|}{4}\right)\right]^{2} < 2 c
    \end{align}
    for $c \in (0, 1]$, then $\mathtt{W}_{\nu, b, \eta}(z) \in S^{\ast}(q_c)$. 
\end{theorem}

Suppose $\eta=\pm 1$, $b=1$, and $\kappa_s=\nu+3/2$. Now for a fixed $c \in (0, 1]$, the right-hand inequality in \eqref{eqn:hypo-Lem-GStruve} holds for $\nu> -1.06868$, and the left-hand inequality holds for 
$\nu> \nu_1$, where $\nu_1$ is the root of 
\[\left[\,{}_{1}F_{2}\!\left(1; \tfrac{3}{2},\ \nu+\tfrac{3}{2};\;\frac{\sqrt[3]{21}}{4\; \sqrt[3]{4}}\right)\right]^{2} - 1= c \left[\,{}_{1}F_{2}\!\left(1; \tfrac{3}{2},\ \nu+\tfrac{3}{2};\;\frac{1}{4}\right)\right]^{2}.\]
A list of values of $\nu_1$ for a fixed $c$ are tabulated in \Cref{tab:value-nu1}. 
\begin{table}[h]
   \begin{tabular}{|c|c|c|c|}
   \hline
 $c$ & $\nu_1$  &$c$ & $\nu_1$\\ \hline
 0.1 & 4.25508 & 0.6 & -0.633535 \\
 0.2 & 1.34049 &  0.7 & -0.778294 \\
 0.3 & 0.36084 &  0.8 & -0.887662 \\
 0.4 & -0.133344 & 0.9 & -0.973308 \\
 0.5 & -0.432457 &1 & -1.04226 \\
 \hline
\end{tabular}
    \caption{Value of $\nu_1$ for a fixed $c$. }
    \label{tab:value-nu1}
\end{table}

Observed that $S^{\ast}(q_c) \subset S^{\ast}$. In particular, taking $\eta= \pm 1$, $b=1$ and $c=1$ in \Cref{thm:Lem-GStruve}, we have the following special case
\begin{corollary}\label{cor:Lem-GStruve}
The function  $\mathtt{W}_{\nu, 1, \pm 1}(z) \in S^{\ast}(q_1) \subset S^{\ast}$ for $\nu> -1.04226$. 
\end{corollary}

The generalized Struve function $\mathtt{W}_{\nu, b, c}(z)$ studied  in \cite{Yagmur-Orhan, Orhan-Yagmur}. Among various geometric properties, it was shown that $\mathtt{W}_{\nu, b, \eta}(z)$ is starlike  under the condition  
\[
\kappa_s >\frac{9 + \sqrt{ 7}}{24} |\eta|= 0.48524 |\eta| .
\]  
In particular, for  $b = 1$, and $\eta = \pm 1$, this implies that the Struve function $\mathcal{S}(z)=\mathtt{W}_{\nu, 1, 1}(z)$ and the modified Struve $\mathcal{L}(z)=\mathtt{W}_{\nu, 1, -1}(z)$ are starlike if $\nu > -0.953204$. Clearly, \Cref{cor:Lem-GStruve} provides better range of $\nu$ for starlikeness. 
\begin{remark}
  We conclude this section by noting that other parameter-based generalizations of Struve functions, such as those discussed in \cite{Ali-Mondal-Nisar, sarkar2024geometric}, may likewise be analyzed for lemniscate starlikeness, though detailed proofs are omitted here.
\end{remark}

\subsection{Example of Error functions}\label{sec:Appl-error}
The \emph{error function}, denoted by $\operatorname{erf}(z)$, is an important transcendental function that arises in diverse areas such as diffusion theory, statistics, and heat conduction. It is defined by the integral representation.
\begin{equation}\label{eq:erf-def}
\operatorname{erf}(z) = \frac{2}{\sqrt{\pi}} \int_{0}^{z} e^{-t^{2}}\, dt, \qquad z \in \mathbb{C}.
\end{equation}
This function is entire and odd, satisfying $\operatorname{erf}(-z) = -\operatorname{erf}(z)$. The complementary error function is given by $\operatorname{erfc}(z) = 1 - \operatorname{erf}(z)$. Both $\operatorname{erf}(z)$ and $\operatorname{erfc}(z)$ play essential roles in probability and mathematical physics, particularly in the formulation of the Gaussian integral and cumulative distribution functions of the normal law.

By expanding $e^{-t^2}$ in its Maclaurin series and integrating term by term, one obtains the series representation
\begin{equation}\label{eq:erf-series}
\operatorname{erf}(z) = \frac{2}{\sqrt{\pi}} \sum_{n=0}^{\infty} \frac{(-1)^{n} z^{2n+1}}{n!\,(2n+1)}
= \frac{2}{\sqrt{\pi}} \left( z - \frac{z^{3}}{3} + \frac{z^{5}}{10} - \frac{z^{7}}{42} + \cdots \right),
\end{equation}
which converges absolutely for all $z \in \mathbb{C}$. This confirms that $\operatorname{erf}(z)$ is an entire function of order $1$ and type $2$. Its derivative is given by
\[
\operatorname{erf}'(z) = \frac{2}{\sqrt{\pi}} e^{-z^2},
\]
which shows that $\operatorname{erf}(z)$ is monotonic and bounded on the real line, with $\operatorname{erf}(z) \to 1$ as $z \to +\infty$ and $\operatorname{erf}(z) \to -1$ as $z \to -\infty$.

In the context of geometric function theory, we consider the following normalization:

\begin{align}\label{eq:erf-normalizations}
\mathtt{F}_{5}(z) = \frac{\sqrt{z} \sqrt{\pi}}{2 }\operatorname{erf}(\sqrt{z})=  \sum_{n=1}^{\infty} \frac{(-1)^{n-1} z^{n}}{(n-1)!\,(2n-1)}.
\end{align}

Now, for $\mathtt{F}_{5}$, the right-hand side of hypothesis \eqref{eqn:hypo-1-thm1} is equivalent to 
\begin{align}
  \sum_{n=1}^{\infty} \sum_{k=0}^n \frac{1}{k!\,(2k+1)} \frac{1}{(n-k)!\,(2(n-k)+1)}.   
\end{align}
 Use of the identity \eqref{eqn:self-CP} gives
\begin{align}
  \sum_{n=1}^{\infty} \sum_{k=0}^n \frac{1}{k!\,(2k+1)} \frac{1}{(n-k)!\,(2(n-k)+1)}= \left(\sum_{n=0}^{\infty} \frac{1}{n!\,(2n+1)} \right)^2 -1=1.13935>1. 
\end{align}
Thus, \Cref{thm:anlytic-LS} is not applicable on $\mathtt{F}_{5}$.

Recall, $\mathfrak{H} \ast \mathfrak{K}$, the well-known Hadamard product of two analytic functions 
 $\mathfrak{H}(z) := \sum_{k=0}^{\infty} a_k z^k $ and $\mathfrak{K}(z):= \sum_{k=0}^{\infty} b_k z^k $. Then 
\begin{align*}
   (\mathfrak{H} \ast \mathfrak{K})(z )= \sum_{k=0}^{\infty} a_k b_k z^k.
\end{align*}
Let us now consider the function $\mathtt{F}_{6}(z)=\mathtt{F}_{5}(z) \ast  (e^z-1)$. Clearly, 
\begin{align}\label{eq:erf-normalizations}
\mathtt{F}_{6}(z) =   \sum_{n=1}^{\infty} \frac{(-1)^{n-1} z^{n}}{(n-1)!\,(2n-1) n!}.
\end{align}
In this case, the numerical sum indicates that 
the right-hand side of hypothesis \eqref{eqn:hypo-1-thm1} is equivalent to 
\begin{align*}
  &\sum_{n=1}^{\infty} \sum_{k=0}^n \frac{1}{k!\,(2k+1) (k+1)!} \frac{1}{(n-k)!\,(2(n-k)+1) (n-k+1)!}\\
  &=  \left(\sum_{n=0}^{\infty} \frac{1}{n!\,(2n+1)(n+1)!} \right)^2 -1=0.402721. 
\end{align*}
Similarly, for $c=1$, the right-hand side of hypothesis \eqref{eqn:hypo-thm1} is equivalent to 
\begin{align*}
  &\sum_{n=1}^{\infty} \sum_{k=0}^n \frac{21^{\tfrac{k}{3}}}{4^{\tfrac{k}{3}}\, k!\,(2k+1) (k+1)!} \frac{21^{\tfrac{n-k}{3}}}{4^{\tfrac{n-k}{3}}\,(n-k)!\,(2(n-k)+1) (n-k+1)!} \\
  &\quad =  \left(\sum_{n=0}^{\infty} \frac{21^{\tfrac{n}{3}}}{4^{\tfrac{n}{3}}\, n!\,(2n+1)(n+1)!} \right)^2 =0.407896. 
\end{align*}
Hence, we have the following result. 
\begin{theorem}
    The function $\mathtt{F}_{6}(z) \in S^{\ast}(q_1)$. In general, $\mathtt{F}_{6}(z) \in S^{\ast}(q_c)$ for $ c_0 \approx 0.577889 < c \le 1$. Here, $c_0$ is the root of the equation 
    \[\, _1F_2\left(\frac{1}{2};\frac{3}{2},2;\frac{\sqrt[3]{21}}{2^{2/3}}\right){}^2-c \, _1F_2\left(\frac{1}{2};\frac{3}{2},2;1\right){}^2=1.\]
\end{theorem}

\section{Conclusion} \label{sec:conclusion}
We conclude by noting that several other normalized special functions may also be included in the class $s^{\ast}(q_c)$. 
 For the sake of brevity, and since their analysis proceeds through similar arguments, we omit the detailed statements and proofs. A few notable examples are listed below:

\begin{enumerate}
    \item The product of Bessel and modified Bessel \cite{Baricz-product-Bessel},

\item Lommel function \cite{MR3596935} and Dini function \cite{Baricz-Deniz-Dini},
\item Wright functions  \cite{JKP-2, Das} and generalized Mittag-Leffler function \cite{Baricz-Anuja}, 
\item $q$-Bessel functions\cite{qb-Baricz1}.
\end{enumerate}

Building upon the findings of \cite{Zayed-Bulboacua} concerning the lemniscate convexity of generalized Bessel functions, one may generalize the underlying ideas to a wider class of normalized analytic functions. Such an extension can be effectively applied to the special functions analyzed in this work. The authors anticipate that an improvement of certain inequalities used in \cite{Zayed-Bulboacua} would lead to more accurate characterizations of lemniscate convex functions. This refinement will be considered as part of future research. 

\section*{Decleration}
\subsection*{Competing interests:} {The authors declare that they have no competing interests.}
\subsection*{Funding:} {Not Applicable.}
\subsection*{Availability of data and materials:} {No datasets were generated or analysed during the current study.}
\subsection*{Author's contribution:} {All authors contributed equally to the writing of this paper. They read and approved the final version of the paper.}

\bibliographystyle{unsrt}
\bibliography{references}

\end{document}